\documentclass{amsart}

\usepackage{amsmath,amssymb,verbatim}
\usepackage{amsthm}
\usepackage{enumerate}
\usepackage{url}
\usepackage{mathtools}
\usepackage{color}
\usepackage{wasysym}

\newtheorem{theorem}{Theorem}[section]
\newtheorem{lemma}[theorem]{Lemma}
\newtheorem{corollary}[theorem]{Corollary}
\newtheorem{proposition}[theorem]{Proposition}

\newtheorem{claim}[theorem]{Claim}
\newtheorem{fact}[theorem]{Fact}

\theoremstyle{definition}
\newtheorem{definition}[theorem]{Definition}
\newtheorem{example}[theorem]{Example}
\newtheorem{remark}[theorem]{Remark}
\newtheorem{question}[theorem]{Question}

\def\acl{\operatorname{acl}}

\def\tp{\operatorname{tp}}

\def\Aut{\operatorname{Aut}}
\def\M{\mathbb M}
\def\cL{\mathcal L}
\def\cK{\mathcal K}
\def\Q{\mathbb Q}

\def\cyc{\operatorname{cyc}}
\newcommand{\seq}{\subseteq}

\newcommand{\eq}{\mathrm{eq}}
\newcommand{\Fraisse}{Fra\"{i}ss\'{e}}
\newcommand{\mand}{\makebox[.5in]{and}}
\newcommand{\mimp}{\makebox[.5in]{$\Rightarrow$}}
\newcommand{\miff}{\makebox[.5in]{$\Leftrightarrow$}}

\newcommand{\bbar}{\bar{b}}
\newcommand{\SOP}{\mathrm{SOP}}
\newcommand{\NSOP}{\mathrm{NSOP}}
\newcommand{\NTP}{\mathrm{NTP}}
\newcommand{\Th}{\mathrm{Th}}
\newcommand{\TOG}{T_{\mathsf{OG}}}

\def\Ind{\setbox0=\hbox{$x$}\kern\wd0\hbox to 0pt{\hss$\mid$\hss}
\lower.9\ht0\hbox to 0pt{\hss$\smile$\hss}\kern\wd0}

\def\Notind{\setbox0=\hbox{$x$}\kern\wd0\hbox to 0pt{\mathchardef
\nn=12854\hss$\nn$\kern1.4\wd0\hss}\hbox to
0pt{\hss$\mid$\hss}\lower.9\ht0 \hbox to 0pt{\hss$\smile$\hss}\kern\wd0}

\def\ind{\mathop{\mathpalette\Ind{}}}
\def\nind{\mathop{\mathpalette\Notind{}}}

\renewcommand\emptyset{\varnothing}

\newcommand{\ter}[1]{\ind^{\!\!{#1}}}
\newcommand{\nter}[1]{\nind^{\!\!{#1}}}
\newcommand{\trt}[1]{\ind^{\!\!\textnormal{#1}}}
\newcommand{\ntrt}[1]{\nind^{\!\!\textnormal{#1}}}

\AtBeginDocument{%
   \def\MR#1{}
}

\title{Three surprising instances of dividing}

\date{April 4, 2024}

\author[G. Conant]{Gabriel Conant}

\address{Department of Mathematics\\
The Ohio State University\\
Columbus, OH 43210\\
 USA}
\email{conant.38@osu.edu}

\author[A. Kruckman]{Alex Kruckman}

\address{Department of Mathematics and Computer Science\\
Wesleyan University\\
Middletown, CT 06459\\
 USA}
\email{akruckman@wesleyan.edu}

\begin{document}

\begin{abstract}
We give three counterexamples to the folklore claim  that in an arbitrary theory, if a complete type $p$ over a set $B$ does not divide over $C\seq B$, then no extension of $p$ to a complete type over $\acl(B)$  divides over $C$.
Two of our examples are also the first known theories where all sets are extension bases for nonforking, but forking and dividing differ for complete types (answering a question of Adler). One example is an $\NSOP_1$ theory with a complete type that forks, but does not divide, over a model (answering a question of d'Elb\'{e}e). Moreover, dividing independence  fails to imply M-independence in this example (which refutes another folklore claim). In addition to these counterexamples, we summarize various related properties of dividing that are still true. We also address consequences for previous literature, including  an earlier unpublished result about forking and dividing in free amalgamation theories, and some claims about dividing in the theory of generic $K_{m,n}$-free incidence structures.
\end{abstract} 

\maketitle

\section{Introduction}

The basis for this paper is the discovery that a certain well-known property of dividing independence, which was originally stated in \cite{Adgeo} and appears as folklore in a number of places, is actually false.

To elaborate, let $T$ be a complete first-order theory with monster model $\M$. We let $\trt{d}$ denote dividing independence for (small) sets in $\M$ (see Definition \ref{def:relations}).  Given $C\subset\M$, we say that $\trt{d}$ has \textbf{algebraic extension over $C$} if, for any $A,B\subset\M$, $A\trt{d}_C B$ implies $A\trt{d}_C \acl(BC)$. We say that $\trt{d}$ has \textbf{algebraic extension} if this holds over any $C\subset\M$. 

Remark 5.4(3) of \cite{Adgeo} states that $\trt{d}$ always has algebraic extension. This also appears as Exercise 1.24$(iii)$ in Chapter 1 of Adler's thesis \cite{Adthesis} (later published as  \cite{Adgeo}).  A solution to this exercise is given in \cite{Adthesis}; however, a gap in the argument was found by the first-named author and Terry in February 2013. At that time, Adler   suggested a new proof (via personal communication), which was rewritten in notes posted on the first-named author's website. Then, in February of 2021, the second-named author found that the same gap was still present in the new proof, but hidden in a more subtle way, and subsequently discovered one of  three counterexamples we will present here. Each of these examples demonstrates unique aspects of how and where algebraic extension for $\trt{d}$ can fail. To help provide context for how these examples have been curated, note that $\trt{d}$ does have algebraic extension over $C$ whenever forking and dividing over $C$ are the same for complete types (see Remark \ref{rem:ext}). So, for example, $\trt{d}$ has algebraic extension in any simple theory. More generally, if $T$ is $\NTP_2$ and $C$ is an extension base for nonforking, then $\trt{d}$ has algebraic extension over $C$ by \cite{ChKa}. With these facts in mind, we now give an overview of our examples.

\begin{enumerate}[$(1)$]
\item In Section \ref{sec:DCO}, we show that $\trt{d}$ fails algebraic extension over $\emptyset$ in $T=\Th(\Q,\cyc)^{\eq}$ where $\cyc$ is the circular order on $\Q$. Note that $T$ is NIP, and that $\emptyset$ is (necessarily) \emph{not} an extension base for nonforking in $T$.
\item In Section \ref{sec:GBF}, we show that $\trt{d}$ fails algebraic extension \emph{over a model} in the $\NSOP_1$ theory $T=(T^\emptyset_{f})^{\eq}$, where $T^\emptyset_{f}$ is the model completion of the empty theory in a language with only a binary function symbol $f$. In this case, the failure of algebraic extension for $\trt{d}$ actually arises through the  failure of $\trt{d}\Rightarrow \trt{M}$ (see Section \ref{sec:Mind} for further discussion of M-independence). Thus this example also refutes the claim made in Remark 5.4(4) of \cite{Adgeo}. We also show that in $T$, all sets are extension bases for nonforking.
\item In Section \ref{sec:OG}, we construct a theory, called $\TOG$, where $\trt{d}$ fails algebraic extension for the  stronger reason that $\trt{d}_C$ need not imply $\trt{d}_{\acl(C)}$ (see Remark \ref{rem:add-to-base} for further discussion). $\TOG$ is an $\NSOP_4$ theory (with $\SOP_3$) in which all sets are extension bases for nonforking. This is the original counterexample discovered by the second-named author. 
\end{enumerate}

In addition to refuting the erroneous claims made in \cite[Remark 5.4]{Adgeo}, these counterexamples also answer some other open questions. First, recall that if $\trt{f}=\trt{d}$  (i.e.,  forking and dividing coincide for complete types), then all sets are extension bases for nonforking. Question A.1 of Adler's thesis \cite{Adthesis} asks whether the converse holds. As noted above, if $\trt{d}$ fails algebraic extension, then $\trt{f}\neq \trt{d}$, and thus both $(2)$ and $(3)$ provide counterexamples to Adler's question. Moreover, $(2)$ and $(3)$ appear to be the first known examples of $\NSOP$ theories in which $\trt{f}\neq \trt{d}$. Whether $\trt{f}=\trt{d}$ in $\NSOP_1$ theories specifically had been discussed by a number of people in the field, and was asked by d'Elb\'{e}e in \cite[Question 2]{dElbACFG}. Example $(2)$  shows that $\trt{f}=\trt{d}$ need not hold in $\NSOP_1$ theories, even over models. In fact, since the discovery of this example, two others have surfaced in previous literature from erroneous arguments relying on \cite[Remark 5.4]{Adgeo} (see Sections \ref{sec:Tmn} and \ref{sec:ACFG}).

We now give an outline of the paper. In Section \ref{sec:prelim} we recall various definitions, and then we spend some time discussing positive results related to \cite[Remark 5.4]{Adgeo} that are still true. For example, algebraic extension can be viewed as one part of the more general question of preservation of algebraic closure for a ternary relation (Definition \ref{def:PAC}). We clarify in Proposition \ref{prop:PAC-AE-div} what amount of preservation one can obtain for $\trt{d}$, and we show that $\trt{f}$ always preserves algebraic closure. These facts are known in the folklore, but for obvious reasons we think it appropriate to provide details. Then in Section \ref{sec:Mind} we focus on the question of when $\trt{d}$ implies $\trt{M}$, which was the main motivation for \cite[Remark 5.4]{Adgeo}. Among other quick observations, we show in Proposition \ref{prop:preg} that $\trt{d}$ implies $\trt{M}$ in  pregeometric theories (where $\trt{M}$ is especially meaningful; see Fact \ref{fact:preg}). We then move on to our three counterexamples summarized above, which are given in Section \ref{sec:examples}. The rest of the paper is devoted to addressing uses of \cite[Remark 5.4]{Adgeo} in previous literature. First, in Section \ref{sec:FAT}, we focus on free amalgamation theories, defined by the first-named author in \cite{CoFA}. It turns out that the theory $\TOG$ in  example $(3)$ above is a free amalgamation theory with disintegrated algebraic closure (Corollary \ref{prop:OGFAT}). Consequently, the failure of algebraic extension in $\TOG$ refutes the claim made in an earlier unpublished research note (of the authors) that $\trt{f}=\trt{d}$ in any such theory. However, the arguments from that note can be adjusted to show that forking and dividing are the same for complete types over algebraically closed sets (see Theorem \ref{thm:FAT}). Finally, Section \ref{sec:corrections} contains brief discussion of  other places where \cite[Remark 5.4]{Adgeo} was used. For example, we amend some incorrect statements from \cite{CoKr} about dividing in the generic theory of $K_{m,n}$-free incidence structures.

\subsection*{Acknowledgements} The authors thank Christian d'Elb\'{e}e and Nick Ramsey for stimulating discussions, as well as the anonymous referee for their careful reading of the paper and helpful revisions. Conant was partially supported by NSF grant DMS-2204787.

\section{Dividing and Algebraic Closure}\label{sec:prelim}

Throughout this section, we work in the setting of a complete theory $T$ with monster model $\M$.  We allow letters $a,b,c,\ldots$ to denote tuples from $\M$ (which may be infinite), but sometimes also use vector notation $\overline{a},\overline{b},\overline{c},\ldots$ when the distinction between tuples and singletons is important.

\subsection{Preliminaries}

We first recall several axioms of a ternary relation $\ind$ on small subsets of $\M$.

\begin{definition}$~$
\begin{enumerate}[$(1)$]
\item \textit{(invariance)} For all $A,B,C$, if $A\ind_C B$ and $\sigma\in\Aut(\M)$ then $\sigma(A)\ind_{\sigma(C)}\sigma(B)$.
\item \textit{(monotonicity)} For all $A,B,C$, if $A\ind_C B$ then $A'\ind_C B'$ for all $A'\seq A$ and $B'\seq B$.
\item \textit{(base monotonicity)} For all $A,B,C$, if $A\ind_C B$ and $D\seq B$ then $A\ind_{CD} B$. 
\item \textit{(extension)} For all $A,B,C$, if $A\ind_C B$ and $D\supseteq B$, then there is some $A'\equiv_{BC} A$ such that $A'\ind_C D$. 
\item \textit{(existence)} $A\ind_C C$ for all $A$ and $C$.
\end{enumerate}
\end{definition}

Next we recall a few special independence relations. 

\begin{definition}\label{def:relations}
Fix $A,B,C\subset \M$, and let $a$ be a tuple enumerating $A$.
\begin{enumerate}[$(1)$]
    \item $A\trt{d}_C B$ if $\tp(a/BC)$ does not divide over $C$.\vspace{3pt}
    \item $A\trt{f}_C B$ if $\tp(a/BC)$ does not fork over $C$.\vspace{3pt}
    \item $A\trt{a}_C B$ if $\acl(AC)\cap \acl(BC)=\acl(C)$.\vspace{3pt}
    \item $A\trt{M}_C B$ if $\acl(AD)\cap \acl(BD)=\acl(D)$ for any $D$ such that $C\seq D\seq\acl(BC)$.
\end{enumerate}
\end{definition}

For complete definitions of forking and dividing, see \cite{Adgeo} or \cite[Chapter 7]{TZ}. Recall also that a set $C\subset\M$ is called an \emph{extension base for nonforking} if $A\trt{f}_C C$ for any $A\subset\M$. Thus $\trt{f}$ satisfies existence if and only if all sets are extension bases for nonforking.

\begin{remark}
    In Adler's thesis \cite{Adthesis}, ``existence" is used for a stronger axiom that is later renamed ``full existence" in \cite{Adgeo}. So to forestall potential confusion, we note that existence and full existence are equivalent for $\trt{f}$. (See also \cite[Fact 4.2]{CoHa})
\end{remark}

\subsection{Algebraic extension and preservation of algebraic closure}

The next definition repeats the  key notion from the introduction, but in the  the setting of a general  ternary relation $\ind$ on small subsets of $\M$. 

\begin{definition}
    A ternary relation $\ind$ has \textbf{algebraic extension over $C\subset\M$} if, for any $A,B\subset\M$, $A\ind_C B$ implies $A\ind_C \acl(BC)$.
    We say $\ind$ has \textbf{algebraic extension} if this holds over any $C\subset\M$. 
\end{definition}

We note that algebraic extension is nearly the same as the \emph{right closure} axiom defined by d'Elb\'{e}e  in \cite{dElbACFG,dElbee}. Indeed, the two are equivalent for ternary relations satisfying right monotonicity and \emph{right normality} ($A\ind_C B$ implies $A\ind_C BC$). These axioms are also related to the following stronger property.

\begin{definition}\label{def:PAC}
    A ternary relation $\ind$ \textbf{preserves algebraic closure} if, for any $A,B,C\subset\M$,
    \[
    \textstyle A\ind_C B ~\Leftrightarrow~ \acl(AC)\ind_C B~\Leftrightarrow~A\ind_C \acl(BC)~\Leftrightarrow~A\ind_{\acl(C)} B.
    \]
\end{definition}

\begin{lemma}\label{lem:PAC-AE-test}
    Let $\ind$ be an invariant ternary relation satisfying monotonicity and base monotonicity. Then $\ind$ preserves algebraic closure if and only if it has algebraic extension and satisfies the following axioms.
    \begin{enumerate}[$(i)$]
    \item For any $A,B,C\subset \M$, if $A\ind_C B$ then $\acl(AC)\ind_C B$. 
    \item For any $A,B,C\subset \M$, if $A\ind_{\acl(C)} B$ then $A\ind_C B$. 
    \end{enumerate} 
\end{lemma}
\begin{proof}
We only need to prove the nontrivial direction. Assume $\ind$ has algebraic extension and satisfies $(i)$ and $(ii)$. Then in order to establish preservation of algebraic closure for $\ind$, it suffices to prove:
\begin{enumerate}
\item[$(iii)$] For any $A,B,C\subset \M$, if $\acl(AC)\ind_C B$ then $A\ind_C B$.
\item[$(iv)$] For any $A,B,C\subset \M$, if $A\ind_C \acl(BC)$ then $A\ind_{\acl(C)} B$.
\end{enumerate}
Both of these are immediate from monotonicity and base monotonicity for $\ind$. 
\end{proof}

The reason we have not given axioms $(i)$ and $(ii)$ their own names is because our focus is on $\trt{d}$ which, as we observe next, always satisfies these axioms.

\begin{proposition}\label{prop:PAC-AE-div}
    In any theory, $\trt{d}$ satisfies axioms $(i)$ and $(ii)$ of Lemma \ref{lem:PAC-AE-test}. Hence $\trt{d}$ preserves algebraic closure if and only if it has algebraic extension.
\end{proposition}
\begin{proof}
    The second claim follows from the first by Lemma \ref{lem:PAC-AE-test} and the fact that $\trt{d}$ satisfies monotonicity and base monotonicity. So it remains to show $\trt{d}$ satisfies $(i)$ and $(ii)$ of Lemma \ref{lem:PAC-AE-test}. 

    For axiom $(i)$, recall from \cite[Corollary 7.1.5]{TZ} that $A\trt{d}_C B$ if and only if for any $C$-indiscernible sequence $(b)_{i<\omega}$, with $b_0$ enumerating $B$, there is  an $AC$-indiscernible sequence $(b'_i)_{i<\omega}$ with the same type over $BC$ as $(b_i)_{i<\omega}$. Since any $AC$-indiscernible sequence is automatically $\acl(AC)$-indiscernible (see, e.g, \cite[Exercise 7.1.1]{TZ}), axiom $(i)$ follows.

For axiom $(ii)$, assume $A\trt{d}_{\acl(C)} B$ and let $a$ and $b$ be enumerations of $A$ and $B$, respectively. Suppose $(b_i)_{i<\omega}$ is a $C$-indiscernible sequence with $b_0=b$. Then $(b_i)_{i<\omega}$ is $\acl(C)$-indiscernible. Since $A\trt{d}_{\acl(C)} B$, there is some $a'$ such that $a'b_i\equiv_{\acl(C)}ab$ for all $i<\omega$. Therefore $a'b_i\equiv_{C}ab$ for all $i<\omega$. Thus we have shown $A\trt{d}_C B$.
\end{proof}

\begin{remark}
    It also follows that $\trt{d}$ preserves algebraic closure if and only if, for all $A,B,C\subset\M$, we have
    \[
    \textstyle A\trt{d}_C B~\Rightarrow~\acl(AC)\trt{d}_{\acl(C)}\acl(BC).
    \]
    Indeed, assume the latter condition holds. We show $\trt{d}$ has algebraic extension, and hence preserves algebraic closure by Proposition \ref{prop:PAC-AE-div}. Suppose $A\trt{d}_C B$. By assumption $\acl(AC)\trt{d}_{\acl(C)}\acl(BC)$. Then $A\trt{d}_{\acl(C)}\acl(BC)$ by monotonicity, and thus $A\trt{d}_C\acl(BC)$ by Proposition \ref{prop:PAC-AE-div}.
\end{remark}

\begin{remark}\label{rem:add-to-base}
    Suppose $\ind$ is an invariant ternary relation with monotonicity and base monotonicity (e.g., $\trt{d}$). Consider the following axiom 
    \[
    \textstyle A\ind_C B~\Rightarrow~A\ind_{\acl(C)} B\tag{$\dagger$}
    \]
    Then clearly $(\dagger)$ holds if $\ind$ has algebraic extension. So given that $\trt{d}$ need not always satisfy algebraic extension, it becomes natural to ask whether the weaker axiom $(\dagger)$ holds for $\trt{d}$. The theory constructed in Section \ref{sec:OG} will show that this is not always the case. It is also worth noting that for $\trt{d}$ specifically, $(\dagger)$ is equivalent to the following weakening of algebraic extension: $A\trt{d}_C B\Rightarrow A\trt{d}_CB\acl(C)$.  
\end{remark}

\begin{remark}\label{rem:acl-exist}
    Given $A,C\subset\M$, we always have  $A\trt{d}_C \acl(C)$ (since any $C$-indiscernible sequence is $\acl(C)$-indiscernible). Thus, any failure of algebraic extension over $C$ for $\trt{d}$ must involve a set $B$ not contained in $\acl(C)$. 
\end{remark}

\begin{remark}\label{rem:ext}
In \cite{Adgeo}, algebraic extension is  referred to as a ``weak extension property" since it holds of any invariant ternary relation satisfying extension. Indeed, if $\ind$ satisfies extension and $A\ind_C B$, then there is some $A'\equiv_{BC} A$ such that $A'\ind_C \acl(BC)$, whence $A\ind_C \acl(BC)$  by invariance. 
\end{remark}

Recall (e.g., from \cite[Section 4]{Adgeo}) that forking independence $\trt{f}$ is obtained by ``forcing the extension axiom" on $\trt{d}$. So in light of Proposition \ref{prop:PAC-AE-div} and Remark \ref{rem:ext}, it is reasonable to expect $\trt{f}$ to preserve algebraic closure in any theory. This is again a folklore result (which is actually true in this case), however a proof does not seem to appear in the literature (it is stated for simple theories in \cite[Proposition 5.20]{Casanovas}). So we take the opportunity to provide details.

\begin{proposition}\label{prop:PACfork}
    In any theory,  $\trt{f}$ preserves algebraic closure. 
\end{proposition}
\begin{proof}
     Recall that $\trt{f}$ satisfies  monotonicity, base monotonicity, and extension (hence algebraic extension by Remark \ref{rem:ext}). So it suffices to show that $\trt{f}$ satisfies axioms $(i)$ and $(ii)$ of Lemma \ref{lem:PAC-AE-test}.

    For axiom $(i)$, suppose $A\trt{f}_C B$. In order to show $\acl(AC)\trt{f}_C B$, we need to fix $D\supseteq B$ and find $E\equiv_{BC} \acl(AC)$ such that $E\trt{d}_C D$. Since $A\trt{f}_C B$, there is $A'\equiv_{BC} A$ such that $A'\trt{d}_C D$. So $\acl(A'C)\trt{d}_C D$ by Proposition \ref{prop:PAC-AE-div}. Thus we can take $E=\acl(A'C)$. 

For axiom $(ii)$, suppose $A\trt{f}_{\acl(C)} B$. In order to show $A\trt{f}_C B$, we need to fix $D\supseteq B$ and find $A'\equiv_{BC} A$ such that $A'\trt{d}_C D$. Since $A\trt{f}_{\acl(C)} B$, there is $A'\equiv_{B\acl(C)} A$ such that $A'\trt{d}_{\acl(C)} D$. Then $A'\equiv_{BC}A$, and $A'\trt{d}_C D$ by Proposition \ref{prop:PAC-AE-div}.
\end{proof}

\begin{corollary}
    If forking and dividing are the same for complete types (e.g., if $T$ is simple), then $\trt{d}$ preserves algebraic closure.
\end{corollary}

\subsection{M-independence}\label{sec:Mind}

We now turn our focus to the relationship between dividing independence and M-independence. In \cite{Adgeo}, Adler uses the erroneous Remark 5.4(3) to conclude Remark 5.4(4), which says that $\trt{d}$ implies $\trt{M}$ in any theory.  The theory constructed in Section \ref{sec:GBF} will in fact also serve as a counterexample to \cite[Remark 5.4(4)]{Adgeo} (see Corollary \ref{cor:GBFM}). However, we should note that Adler's primary motivation for the entirety of \cite[Remark 5.4]{Adgeo} is to obtain $\trt{f}\Rightarrow\trt{\thorn}$. This implication, which is the more important one for the purposes of developing thorn-forking, is still true, since $\trt{f}$ satisfies extension.

\begin{remark}\label{rem:dtoa}
Although the implication $\trt{d}\Rightarrow\trt{M}$ does not hold in general, the weaker implication $\trt{d}\Rightarrow\trt{a}$ is always true. However, proofs of this fact in the literature often quote \cite[Remark 5.4(3)]{Adgeo}. This is further  discussed  in \cite[Section 2]{CoHa}, where a direct proof of $\trt{d}\Rightarrow\trt{a}$ is given using P. M. Neumann's Lemma (see \cite[Proposition 2.3]{CoHa}; the argument is similar to a Mathematics Stack Exchange post of the second-named author \cite{KrMSE1}). As noted in \cite[Remark 4.12]{CoHa}, there is another quick proof of $\trt{d}\Rightarrow\trt{a}$ using extension for $\trt{a}$. This argument is similar to another Mathematics Stack Exchange post of the second-named author \cite{KrMSE2}, and we copy it here for completeness.

 Assume $A\trt{d}_C B$, and let $a$ and $b$ enumerate $A$ and $B$. By extension for $\trt{a}$ and Erd\H{o}s-Rado, one can construct a $C$-indiscernible sequence $(b_i)_{i<\omega}$ such that $b_0=b$ and $b_i\trt{a}_C b_{<i}$ for all $i<\omega$. By \cite[Corollary 7.1.5]{TZ} (and \cite[Exercise 7.1.1]{TZ}), there is an $\acl(a C)$-indiscernible sequence $(b'_i)_{i<\omega}$ such that $(b'_i)_{i<\omega}\equiv_{b_0 C}(b_i)_{i<\omega}$. Therefore
 \begin{align*}
 \acl(a C)\cap \acl(b C) &= \acl(a C)\cap \acl(b'_0 C)&\text{ (since $b'_0=b_0=b$)}\\
 &\seq \acl(b'_1 C)\cap\acl(b'_0C) &\text{ (since $b'_0\equiv_{\acl(a C)}b'_1$)}\\
 &=\acl(C) &\text{ (since $b'_0 b'_1\equiv_C b_0 b_1$ and $\textstyle b_1\trt{a}_C b_0$).}
 \end{align*} 
 So $A\trt{a}_C B$, as desired.

A direct proof of extension for $\trt{a}$ is given by Adler in \cite[Proposition 1.5(1)]{Adgeo}. A shorter proof using P. M. Neumann's Lemma can be found in \cite[Proposition 2.1]{CoHa}. (Both \cite{Adgeo} and \cite{CoHa} actually focus on ``full existence" for $\trt{a}$; see  \cite[Fact 4.2]{CoHa}.) 
\end{remark}

\begin{remark}\label{rem:little-m}
    Recall that $\trt{M}$ is often viewed as the result of ``forcing base monotonicity" on $\trt{a}$. We can now see that this is slightly misleading. In particular, it would be desirable for the ternary relation obtained by forcing base monotonicity on $\trt{a}$ to be weaker than any other relation that has base monotonicity and implies $\trt{a}$. But this is not the case for $\trt{M}$, since it is not always weaker than $\trt{d}$ (see Section \ref{sec:GBF}). On the other hand, in \cite[Corollary 4.8]{Adgeo}, Adler refers to a variation of M-independence, which is defined as $A\trt{m}_C B$ if and only if $A\trt{a}_D B$ for all $C\seq D\seq BC$. By definition, $\trt{m}$ has base monotonicity, implies $\trt{a}$, and is weaker than any other relation with those properties (so, e.g., $\trt{d}\Rightarrow\trt{m}$ always holds). Altogether, $\trt{M}$ is really the result of first forcing base monotonicity on $\trt{a}$ to obtain $\trt{m}$, and then forcing algebraic extension on $\trt{m}$. 
\end{remark}

We will now examine some special cases in which one can recover the implication from $\trt{d}$ to $\trt{M}$.

\begin{remark}$~$
\begin{enumerate}
    \item If $\trt{d}$ has algebraic extension over $C\subset\M$ then $\trt{d}_C\Rightarrow \trt{M}_C$. 
    \item It is straightforward to show that $\trt{d}$ always preserves \emph{definable} closure.  Thus if algebraic closure coincides with definable closure, then $\trt{d}$ has algebraic extension and hence $\trt{d}\Rightarrow\trt{M}$.  
    \item If algebraic closure in $T$ is modular then $\trt{d}\Rightarrow\trt{M}$ since, in this case, $\trt{M}$ coincides with $\trt{a}$ (see \cite[Proposition 1.3$(3)$]{Adgeo}). However, the example in Section \ref{sec:OG} has disintegrated  algebraic closure. So it is not the case that if algebraic closure in $T$ is modular then $\trt{d}$ has algebraic extension. 
\end{enumerate}
\end{remark}

Recall that a theory $T$ is \emph{pregeometric} if algebraic closure in $T$ satisfies the exchange property. In this case, given a  tuple $\overline{a}$ from $\M$ and a set $C\subset\M$, we let $\dim(\overline{a}/C)$ denote the size of an $\acl$-basis for $\overline{a}$ over $C$. In pregeometric theories,  M-independence is especially meaningful due to the following characterization. 

\begin{fact}\label{fact:preg}
    Suppose $T$ is pregeometric. Then given $A,B,C\subset\M$, $A\trt{M}_C B$ if and only if  $\dim(\overline{a}/BC)=\dim(\overline{a}/C)$ for any finite tuple $\overline{a}$ from $A$. Moreover, $\trt{M}$ is a strict independence relation in the sense of \cite{Adgeo}.
\end{fact}

As far as we are aware, an explicit statement and proof of Fact \ref{fact:preg} first appeared in unpublished notes by Adler \cite{Adpreg}. See also Theorems 1.2.12 and 1.2.15 of \cite{dElbee}.

\begin{proposition}\label{prop:preg}
    If $T$ is pregeometric then $\trt{d}$ implies $\trt{M}$.
\end{proposition}
\begin{proof}
    Assume $T$ is pregeometric. Fix $A,B,C\subset\M$ such that $A\ntrt{M}_C B$. Then there is some set $D$ such that $C\seq D\seq \acl(BC)$ and $A\ntrt{a}_D B$. So there is some singleton $e\in (\acl(AD)\cap \acl(BD))\backslash \acl(D)$. Let $\overline{a}$ be a finite tuple from $A$ of minimal length such that $e\in\acl(\overline{a}D)$. Note that $\overline{a}$ is nonempty (since $e\not\in\acl(D)$) and algebraically independent over $D$. Write $\overline{a}=a_0\overline{a}'$ where $a_0$ is a singleton. Then $e\in\acl(a_0\overline{a}'D)\backslash\acl(\overline{a}'D)$, hence by exchange $a_0\in\acl(e\overline{a}'D)\seq \acl(\overline{a}'BC)$. Now let $\overline{b}$ be a finite tuple from $B$ of minimal length such that $a_0\in\acl(\overline{a}'\overline{b}C)$. Again, $\overline{b}$ is nonempty (since $\overline{a}$ is algebraically independent over $D$, and hence over $C$ as well) and algebraically independent over $\overline{a}'C$. Write $\overline{b}=b_0\overline{b}'$ where $b_0$ is a singleton. Then $a_0\in \acl(\overline{a}'b_0\overline{b}'C)\backslash \acl(\overline{a}'\overline{b}'C)$, hence by exchange $b_0\in\acl(\overline{a}\overline{b}' C)$. Recall that $\overline{b}'$ is algebraically independent over $\overline{a}'C$, hence over $C$, so $b_0\not\in\acl(\overline{b}'C)$. Altogether, $b_0$ witnesses $A\ntrt{a}_{C\overline{b}'}B$. Therefore $A\ntrt{d}_{C\overline{b}'}B$. By base monotonicity for $\trt{d}$, and since $\bbar'$ is a tuple from $B$, we conclude $A\ntrt{d}_C B$, as desired.
\end{proof}

\begin{remark}
    The previous proof actually shows that if $T$ is pregeometric then $\trt{M}=\trt{m}$ (see Remark \ref{rem:little-m}), hence $\trt{m}$ has algebraic extension. However the example in Section \ref{sec:OG} is pregeometric, and thus is not the case that $\trt{d}$ has algebraic extension in an arbitrary pregeometric theory.
\end{remark}

\subsection{Conclusion}

In the spirit of Adler's work, a natural thing to do at this point is to construct a new ternary relation by forcing algebraic extension on $\trt{d}$. So we define: 
\[
\textstyle A\trt{da}_C B~\Leftrightarrow~ A\trt{d}_C \acl(BC).
\]
Then $\trt{da}$ satisfies all of the properties that Remarks 5.4(3) and 5.4(4) of \cite{Adgeo} intended for $\trt{d}$. In particular, $\trt{da}$ preserves algebraic closure and we have 
\[
\textstyle \trt{f}~\Rightarrow~\trt{da}~\Rightarrow~\trt{M}.
\]
Note also that $\trt{da}$ satisfies existence (by Remark \ref{rem:acl-exist}). Moreover, $\trt{f}=\trt{da}$ if and only if forking and dividing are the same for complete types over algebraically closed sets.

Altogether, perhaps $\trt{da}$ is the more natural notion of ``dividing independence" in an arbitrary theory. For example, by replacing $\trt{d}$ with $\trt{da}$, we obtain the following variations of the (formerly) open questions discussed in the introduction.

\begin{question}\label{Q:new}
Let $T$ be a complete theory.
\begin{enumerate}
    \item Suppose $\trt{f}$ satisfies existence. Do $\trt{f}$ and $\trt{da}$ coincide?
    \item Suppose $T$ is $\NSOP$ (or even just $\NSOP_1$). Do $\trt{f}$ and $\trt{da}$ coincide (even just over models)?
\end{enumerate}
\end{question}

Implicit in these questions is the third question of whether there is an $\NSOP$ theory in which $\trt{f}$ fails existence.  In fact, this is open even for $\NSOP_1$ theories, and is a question of notable significance for this region   due to the results of \cite{DKR}. We also note that for $T$ $\NSOP_1$,  the second question is a stronger version of \cite[Question 9.19]{KaRam}, which asks whether $A\trt{f}_M N\Leftrightarrow A\trt{d}_M N$ for all sets $A$ and models $M\prec N\prec\M\models T$.

\section{The Counterexamples}\label{sec:examples}

\subsection{Dense circular order with unordered pairs}\label{sec:DCO}

In this section, we describe a very straightforward example of the failure of algebraic extension for $\trt{d}$ (refuting \cite[Remark 5.4(3)]{Adgeo}).

Let $T=\Th(\Q,\cyc)^{\eq}$, where 
\[
\cyc(x,y,z)\miff (x<y<z) \vee (z<x<y)\vee (y<z<x).
\]
Let $O$ be the home sort, and let $P$ be the sort corresponding to the definable equivalence relation $E$ defined by $(x,y)E(x',y')$ if and only if $\{x,y\} = \{x',y'\}$. So the elements of $P$ can be identified with the unordered pairs and singletons from $O$. Let $q\colon O^2\to P$ be the quotient map. Working in $\M\models T$, let $a,d_1,d_2$ be three distinct elements of the home sort with $\cyc(d_1,a,d_2)$, and let $b=q(d_1,d_2)$.

\begin{claim}
$a\trt{d}_\emptyset b$, but $a\ntrt{d}_\emptyset\acl(b)$.
\end{claim}
\begin{proof}
First, since $d_1,d_2\in\acl(b)$, and the formula $\cyc(d_1,x,d_2)$ divides over $\emptyset$, we have $a\ntrt{d}_\emptyset\acl(b)$. So we need to show $a\trt{d}_\emptyset b$. We start with a general remark. Suppose $a'$, $d'_1$, $d'_2$ are any pairwise distinct elements of $O$, and set $b'=q(d'_1,d'_2)$. Then either $\cyc(d'_1,a',d'_2)$, in which case we have an isomorphism $ad_1d_2\to a'd'_1d'_2$, or $\cyc(d'_2,a',d'_1)$, in which case we have an isomorphism $ad_1d_2\to a'd'_2d'_1$. In either case, by quantifier elimination this yields  $a'b'\equiv_\emptyset ab$. 

Now let $(b_n)_{n<\omega}$ be an indiscernible sequence in $\tp(b/\emptyset)$. For each $n$, let $b_n=q(d^n_1,d^n_2)$, and note that $d^n_1\neq d^n_2$. Choose  $a'$ from the home sort distinct from $d^n_1,d^n_2$ for all $n<\omega$. Then, by the above, we have $a'b_n\equiv_\emptyset ab$ for all $n<\omega$, so $a\trt{d}_\varnothing b$, as desired. 
\end{proof}

We conclude that in $T$, $\trt{d}$ fails algebraic extension over $\emptyset$. To recap the discussion from the introduction, note that $T$ is an NIP theory, and that $\emptyset$ is (necessarily) not an extension base for nonforking.

\subsection{Generic binary function with unordered pairs}\label{sec:GBF}

In this section, we give another example of a theory in which $\trt{d}$ fails algebraic extension. Like before, this theory will be the imaginary expansion of a well-known example from the literature. However, in this case we will have better behavior related to forking (compared to the circular order on $\Q$), and thus the failure of algebraic extension for $\trt{d}$ will have more interesting consequences.  

Let $T^\emptyset_f$ be the model completion of the empty theory in a language containing only a binary function symbol $f$. Then $T^\emptyset_f$ is $\NSOP_1$ by \cite{KrRa}, and thus so is the imaginary expansion, which we denote $T^\eq_f$.  We will prove that $T^\eq_f$ has the following properties:
\begin{enumerate}[$(1)$]
\item For any small model $M$, $\trt{d}_M$ does not imply $\trt{M}_M$. So in addition to implying the failure of algebraic extension for $\trt{d}$ (even over models), this also refutes the weaker claim made in \cite[Remark 5.4$(4)$]{Adgeo}.
\item Consequently, $\trt{d}\neq\trt{f}$, even over models, which gives a negative answer to an open question about $\NSOP_1$ theories (discussed in the introduction). 
\item All sets are extension bases for nonforking in $T^\eq_f$.  So this theory also gives a negative answer to \cite[Question A.1]{Adthesis}, which asks whether forking and dividing are always the same for complete types in such a theory.
\end{enumerate}

Toward obtaining these statements, we first recall from Corollaries 3.10 and 3.11 of \cite{KrRa} that $T^\emptyset_f$ has quantifier elimination and the algebraic closure of any set coincides with the substructure generated by that set. Moving now to $T^\eq_f$, we work in a similar setting as in the previous subsection.
Let $O$ be the home sort, and let $P$ be the sort corresponding to the definable equivalence relation $E$ defined by $(x,y)E(x',y')$ if and only if $\{x,y\} = \{x',y'\}$. So as before, the elements of $P$ can be identified with the unordered pairs and singletons from $O$. Let $q\colon O^2\to P$ be the quotient map. We write $\acl^O(x)$ for the elements of the home sort which are algebraic over $x$. 

Let $M$ be a small model of $T^\eq_f$, and write $M^O$ for its home sort. Let $d_1$ and $d_2$ be distinct elements of the home sort, which are not in $M$, such that $f(d_i,d_j) = d_i$ for all $i,j\in \{1,2\}$ and $f(m,d_k) = f(d_k,m) = d_k$ for all $k\in \{1,2\}$ and all $m\in M^O$. Let $b = q(d_1,d_2)$. 

Note that $\acl^O(Mb) = \acl^O(Md_1d_2) = \langle M^O d_1 d_2\rangle = M^O\cup \{d_1,d_2\}$, and the map swapping $d_1$ and $d_2$ and fixing $M$ pointwise is an automorphism of this substructure. Also, $\acl^O(Md_i) = \langle M^Od_i\rangle = M^O\cup \{d_i\}$ for all  $i\in \{1,2\}$. 

Now let $a$ be an element of the home sort such that $f(a,a) = a$, $f(a,m) = f(m,a) = a$ for all $m\in M$, $f(a,d_1) = f(d_1,a) = a$, and $f(a,d_2) = f(d_2,a) = d_1$. Then $\acl^O(Mad_1) = \langle M^Oad_1\rangle = M^O\cup \{a,d_1\}$, while $\acl^O(Mad_2) = \langle M^Oad_2\rangle = M^O\cup \{a,d_1,d_2\} = \acl^O(Mad_1d_2)$.

\begin{claim}
$a\trt{d}_M b$, but $a\ntrt{M}_M b$ and $a\ntrt{d}_M \acl(Mb)$. 
\end{claim}
\begin{proof}
First note that $d_1\in \acl(Mad_2)\cap \acl(Mb)$, but $d_1\notin \acl(Md_2)$. So $a\ntrt{a}_{Md_2} b$. Since $d_2\in \acl(Mb)$, we have $a\ntrt{M}_M b$ and $a\ntrt{d}_M \acl(Mb)$.

Now we show $a\trt{d}_M b$. Let $a'$ be an element of $O$ and $b'$ an element of $P$, with $b' = q(d_1',d_2')$. By quantifier elimination, $ab\equiv_M a'b'$ if and only if there is an isomorphism $\langle Mad_1d_2\rangle\to \langle Ma'd_1'd_2'\rangle$, fixing $M$ pointwise, with $a\mapsto a'$ and either $d_1\mapsto d_1'$ and $d_2\mapsto d_2'$,   or $d_1\mapsto d_2'$ and $d_2\mapsto d_1'$. 

Let $(b_n)_{n< \omega}$ be an $M$-indiscernible sequence in $\text{tp}(b/M)$. For each $n$, let $b_n = q(d^n_1,d^n_2)$, and note that $d^n_1\neq d^n_2$ and $f(d^n_i,d^n_j) = d^n_i$ for all $i,j\in \{1,2\}$. Without loss of generality, we can enumerate $d^n_1$ and $d^n_2$ so that the sequence $(b_nd^n_1d^n_2)_{n< \omega}$ is $M$-indiscernible. 

Case 1: For all $n\neq m$, $d^n_2 \neq d^m_2$. Then we can find an $a'$ such that $f(a',a') = a'$, $f(a',m) = f(m,a') = a'$ for all $m\in M$, $f(a',d^n_1) = f(d^n_1,a') = a'$, and $f(a',d^n_2) = f(d^n_2,a') = d^n_1$ for all $n$. Then $a'b_n \equiv_M ab$ for all $n$.

Case 2: For all $n\neq m$, $d^n_2 = d^m_2$. In this case, if $d^n_1\neq d^m_1$, we cannot find $a'$ such that $f(a',d^n_2) = d^n_1$ and $f(a',d^m_2) = d^m_1$. However, we can find an $a'$ such that $f(a',a') = a'$, $f(a',m) = f(m,a') = a'$ for all $m\in M$,  $f(a',d^n_2) = f(d^n_2,a') = a'$, and $f(a',d^n_1) = f(d^n_1,a') = d^n_2$ for all $n$. Then $a'b_n \equiv_M ab$ for all $n$ (though in this case the witnessing isomorphisms map $d^n_1\mapsto d_2$ and $d^n_2\mapsto d_1$). 

In either case, we have an $a'$ such that $a'b_n \equiv_M ab$ for all $n<\omega$, so $a\trt{d}_M b$. 
\end{proof}

\begin{corollary}\label{cor:GBFM}
$\trt{d}$ does not necessarily imply $\trt{M}$ in general. Moreover, $\trt{d}$ does not necessarily imply $\trt{f}$ in  $\NSOP_1$ theories, even over models. 
\end{corollary}

Explicitly, in $T^\eq_f$ the formula $\exists w_1\exists w_2\,(q(w_1,w_2) = b\land f(x,w_2) = w_1)$ in $\tp(a/Mb)$ forks but does not divide. It implies the disjunction $f(x,d_2)  = d_1\lor f(x,d_1) = d_2$, and each disjunct divides over $M$.

We have now established statements $(1)$ and $(2)$ above. For $(3)$, we need the following general observation.

\begin{proposition}\label{prop:eqtransfer}
    Suppose $T$ is a complete theory with geometric elimination of imaginaries.
    \begin{enumerate}[$(a)$]
    \item Assume $\trt{f}$ satisfies existence in $T$. Then $\trt{f}$ satisfies existence in $T^{\eq}$.
    \item Assume $\trt{f}=\trt{da}$ in $T$. Then $\trt{f}=\trt{da}$ in $T^{\eq}$.
    \end{enumerate}
\end{proposition}
\begin{proof}
    We prove part $(b)$ and leave $(a)$ (which is similar and easier) to the reader. 
    Fix $A,B,C\seq \M^{\eq}$ and assume $A\trt{da}_C B$ in $T^\eq$. We need to show $A\trt{f}_C B$ in $T^\eq$. Without loss of generality, assume $C\seq A\cap B$. By geometric elimination of imaginaries, there are real sets $A',B',C'\seq\M$ such that $\acl^{\eq}(A)=\acl^{\eq}(A')$, $\acl^{\eq}(B)=\acl^{\eq}(B')$, and $\acl^{\eq}(C)=\acl^{\eq}(C')$. Since $\trt{da}$ preserves algebraic closure, we have $\acl^{\eq}(A')\trt{d}_{\acl^{\eq}(C')}\acl^{\eq}(B')$. By monotonicity for $\trt{d}$ and Proposition \ref{prop:PAC-AE-div}, this yields $\acl(A')\trt{d}_{\acl(C')}\acl(B')$ in $T^{\eq}$, and hence also in $T$. So $\acl(A')\trt{da}_{\acl(C')}\acl(B')$ in $T$, and thus $\acl(A')\trt{f}_{\acl(C')}\acl(B')$ in $T$ by assumption. Therefore $\acl(A')\trt{f}_{\acl(C')}\acl(B')$ in $T^{\eq}$, which yields $A\trt{f}_C B$ in $T^{\eq}$ using Proposition~\ref{prop:PACfork} and the choice of $A'$, $B'$, and $C'$.
\end{proof}

Now, to obtain statement $(3)$ above, we recall the results from \cite{KrRa} that $T^\emptyset_f$ has weak elimination of imaginaries and satisfies $\trt{f}=\trt{d}$ (see Section \ref{sec:KrRam} for further related discussion). So $\trt{f}$ satisfies existence in $T^\eq_f$ by the previous proposition. Note that this conclusion only requires part $(a)$; however part $(b)$ gives us the stronger statement that $\trt{f}=\trt{da}$ in $T^\eq_f$ (c.f. Question \ref{Q:new}).

\begin{remark}
Since $\trt{f}=\trt{d}$ is true in $T^\emptyset_f$ but not in $T^\eq_f$, our example shows that the property $\trt{f} = \trt{d}$ does not transfer from $T$ to $T^\eq$, even when $T$ has weak elimination of imaginaries. This is in contrast to the properties $\trt{f} = \trt{da}$ and existence for $\trt{f}$ in Proposition~\ref{prop:eqtransfer}.
\end{remark}

\subsection{The original counterexample}\label{sec:OG}

In this section, we construct a complete theory  for which $\trt{d}$ does not have algebraic extension due to a failure of the weaker axiom discussed in Remark \ref{rem:add-to-base}. Specifically, there are  $a,b\in \M$ such that $a\trt{d}_\emptyset b$ and $a\ntrt{d}_{\acl(\emptyset)} b$. This is also the original counterexample first discovered by the second-named author to refute \cite[Remark 5.4(3)]{Adgeo}.

The example we construct will be the theory of a Fra\"iss\'e limit for a Fra\"iss\'e class with free amalgamation. We briefly recall the definition of free amalgamation. Given a (many-sorted) language $\cL$ with only relation symbols and constant symbols, and three $\cL$-structures $A$, $B$, $C$, where $C$ is a common substructure of $A$ and $B$, the \emph{free amalgam} of $A$ and $B$ over $C$ is the $\cL$-structure $D$ defined as follows. For each sort $S$, $S(D) = S(A)\sqcup_{S(C)} S(B)$, the disjoint union of $S(A)$ and $S(B)$ over $S(C)$. For each relation symbol $R$ in $\cL$, $R^D = R^A\cup R^B$. That is, no relations hold in $D$, other than those which hold in $A$ and in $B$. The constant symbols in $\cL$ receive the same interpretation in $D$ as in $A$, $B$, and $C$. We say that a class of finite $\cL$-structures has free amalgamation if whenever $A$, $B$, and $C$ are in the class, and $C$ is a common substructure of $A$ and $B$, the free amalgam of $A$ and $B$ over $C$ is again in the class. 

We now describe our example. Consider a language $\cL'$ with:
\begin{itemize}
\item Three sorts: $O$, $G$, and $C$.
\item Two constant symbols $0$ and $1$ of type $C$. 
\item A relation symbol $R$ of type $G\times G \times C$. 
\item A relation symbol $E$ of type $O\times G \times C$. 
\end{itemize}
Let $\mathcal{K}$ be the class of finite $\cL'$-structures satisfying the following conditions:
\begin{enumerate}
\item $0\neq 1$, and for all $c\in C$, $c = 0$ or $c = 1$.
\item The binary relations $R(x,y,0)$ and $R(x,y,1)$ are disjoint graph relations on $G$, i.e., they are each symmetric and anti-reflexive, and for all $v,w\in G$, it is not the case that $R(v,w,0)$ and $R(v,w,1)$. 
\item For all $v,w\in G$ and $c\in C$, if $R(v,w,c)$, then there is no $o\in O$ such that $E(o,v,c)$ and $E(o,w,c)$. 
\end{enumerate}

It is easy to check that $\mathcal{K}$ is a Fra\"{i}ss\'{e} class with free amalgamation. Let $M'$ be the Fra\"{i}ss\'{e} limit of $\mathcal{K}$. 

Now let $\cL$ be the sub-language of $\cL'$ which omits the constant symbols $0$ and $1$.
Let $M$ be the reduct of $M'$ to $\cL$, and set $\TOG=\Th(M)$. We observe the following:
\begin{enumerate}[$(i)$]
\item $\TOG$ is $\aleph_0$-categorical. Indeed, it is a reduct of a Fra\"{i}ss\'{e} limit of a class of finite structures in a finite language with no function symbols.
\item For any $A\subseteq M$, $\acl(A) = A\cup C$. Indeed, $A\cup C$ is the domain of a substructure of $M'$. Since amalgamation in $\mathcal{K}$ is disjoint,  $A\cup C$ is algebraically closed in $M'$, so it remains algebraically closed in the reduct $M$. 
\item There is an automorphism $\sigma$ of $M$ swapping the two elements of $C$. This can be established by a back-and-forth argument.
\item For any finite tuples $a$ and $b$ from $M$, $\tp(a) = \tp(b)$ if and only if there is an isomorphism of $\cL$-structures $f\colon \acl(a)\to \acl(b)$ such that $f(a) = b$. The forward direction follows from $\aleph_0$-homogeneity of $M$. For the converse, suppose $f\colon \acl(a)\to \acl(b)$ is such an isomorphism. If $f$ is the identity on $C$, then it is an isomorphism between $\cL'$-substructures of $M'$, so it extends to an automorphism of $M'$ moving $a$ to $b$, and hence $\tp(a)  = \tp(b)$. If $f$ swaps the two elements of $C$, let $f' = f\circ \sigma^{-1}\colon \sigma(\acl(a))\to \acl(b)$, where $\sigma$ is an automorphism of $M$ as in $(iii)$ above. Then $f'$ is an isomorphism which is the identity on $C$ and such that $f'(\sigma(a)) = b$. Extending $f'$ to an automorphism of $M'$ as above, then pre-composing with $\sigma$, we find an automorphism of $M$ moving $a$ to $b$, so $\tp(a)  = \tp(b)$.
\end{enumerate}

By $\aleph_0$-categoricity, to understand dividing and the properties $\SOP_n$ in $\TOG$, it suffices to work in $M$. We will refer to the elements of $C$ as $0$ and $1$ (even though these constant symbols are not in $\cL$).

\begin{claim}\label{claim:sop}
$\TOG$ is $\SOP_3$ and $\NSOP_4$.
\end{claim} 
\begin{proof}
Since $\Th(M')$ is the theory of an $\aleph_0$-categorical Fra\"{i}ss\'{e} limit with free amalgamation, it is $\NSOP_4$. (This is an unpublished result of R. Patel \cite{PaSOP4}; see also \cite[Section 4]{CoFA}.) Therefore its reduct $\TOG$ is also $\NSOP_4$. We will show that $\TOG$ is $\SOP_3$ using the ``two formula'' formulation (see, e.g., \cite[Claim 2.19]{Sh500}). Let $x$ be a variable of type $O$ and $y,y'$ variables of type $G$. Let $\varphi(x;y,y')$ be the formula $E(x,y,0)$, and let $\varphi'(x;y,y')$ be the formula $E(x,y',0)$. Let $(b_i,b_i')_{i< \omega}$ be a sequence such that $R(b_i',b_j,0)$ if and only if $i< j$. Then for all $n< \omega$, we have $\{\varphi(x;b_i,b_i')\mid i < n\}\cup \{\varphi'(x;b_j,b_j')\mid j\geq n\}$ is consistent by genericity of the Fra\"iss\'e limit and compactness, but for all $i < j$, $\{\varphi'(x;b_i,b_i'),\varphi(x;b_j,b_j')\}$ is inconsistent by condition (3) in the definition of $\cK$.  This establishes $\SOP_3$. 
\end{proof}

We now show that $\TOG$ exhibits the properties claimed at the start of this section.

\begin{claim}
Fix $a\in O$ and $b\in G$ such that $E(a,b,0)$ and $\lnot E(a,b,1)$. Then $a\trt{d}_\varnothing b$, but  $a\ntrt{d}_{\acl(\varnothing)} b$.
\end{claim}
\begin{proof}
We first show $a\trt{d}_\varnothing b$. 
Let $(b_i)_{i< \omega}$ be an indiscernible sequence in $\text{tp}(b/\varnothing)$. 

Case 1: $M\models \lnot R(b_i,b_j,0)$ for all $i\neq j$. Then we can find some $a'\in O$ with $M\models E(a',b_i,0)\land \lnot E(a',b_i,1)$ for all $i< \omega$. The map $f_i\colon \acl(a'b_i)\to \acl(ab)$ which is the identity on $C$ and maps $a'$ to $a$ and $b_i$ to $b$ is an isomorphism of $\cL$-substructures of $M$, so $a'b_i\equiv ab$ for all $i<\omega$ by $(iv)$ above.

Case  2: Otherwise, by indiscernibility, $M\models R(b_i,b_j,0)$ for all $i\neq j$. Since $R(x,y,0)$ and $R(x,y,1)$ are disjoint relations, $M\models \lnot R(b_i,b_j,1)$ for all $i\neq j$. So we can find some $a'\in O$ with $M\models E(a',b_i,1)\land \lnot E(a',b_i,0)$ for all $i<\omega$. Now the map $f_i\colon \acl(a'b_i)\to \acl(ab)$ which swaps the elements of $C$ and maps $a'$ to $a$ and $b_i$ to $b$ is an isomorphism of $\cL$-substructures of $M$, so $a'b_i\equiv ab$ for all $i<\omega$ by $(iv)$ above.

It remains to show $a\ntrt{d}_{\acl(\varnothing)} b$. Note that $\acl(\varnothing) = C$, so $\text{tp}(a/\acl(\varnothing)b)$ contains the formula $E(x,b,0)$. It suffices to show that this formula divides over $C$. 

By $(iv)$ above, all elements of $G$ have the same type over $C$. Let $(b_i)_{i< \omega}$ be a sequence of elements of $G$ such that for all $i\neq j$, $R(b_i,b_j,0)$. Then the set $\{E(x,b_i,0)\mid i< \omega\}$ is $2$-inconsistent, which witnesses dividing.
\end{proof}

In the introduction we also claimed that all sets are extension bases in  $\TOG$. This will be deduced from our general analysis of free amalgamation theories in Section \ref{sec:FAT} (see Corollary \ref{cor:OGfe}). Therefore $\TOG$ is a second example demonstrating a negative answer to \cite[Question A.1]{Adgeo} (in addition to the example in Section \ref{sec:GBF}).

\section{Free Amalgamation Theories}\label{sec:FAT}

In this section, we show that if $T$ is a \emph{free amalgamation theory} (as defined in \cite{CoFA}) with disintegrated algebraic closure, then forking and dividing are the same for complete types over algebraically closed sets.  This proof is from an unpublished research note by the authors, written in May 2017. In the original note, we in fact claimed that this result holds for complete types over \emph{any} set. However, our argument was flawed because we used \cite[Remark 5.4(3)]{Adgeo} to reduce to algebraically closed sets. In fact, the stronger claim fails since the theory $\TOG$ from Section \ref{sec:OG} is a free amalgamation theory with disintegrated algebraic closure (see Proposition \ref{prop:OGFAT}).

Let $T$ be a complete first-order theory with monster model $\M$.  For convenience, and to follow the conventions of \cite{CoFA}, we will call a subset $A\subset \M$ \emph{closed} if $A=\acl(A)$.

\begin{definition}\label{def}
$T$ is a \textbf{free amalgamation theory} if there is an invariant ternary relation $\ind$ on small subsets of $\M$ satisfying monotonicity and the following axioms:
\begin{enumerate}[$(i)$]
\item \textit{(symmetry)} For all $A,B,C$, if $A\ind_C B$ then $B\ind_C A$.
\item \textit{(full transitivity)} For all $A$ and $D\seq C\seq B$, $A\ind_D B$ if and only if $A\ind_C B$ and $A\ind_D C$.
\item \textit{(full existence over closed sets)} For all $B,C\subset\M$ and tuples $a\in\M$, if $C$ is closed then there is $a'\equiv_C a$ such that $a'\ind_C B$.
\item \textit{(stationarity over closed sets)} For all closed $C\subset\M$ and closed tuples $a,a',b\in\M$, with $C\seq a\cap b$, if $a\ind_C b$, $a'\ind_C b$, and $a'\equiv_C a$, then $ab\equiv_C a'b$.
\item \textit{(freedom)} For all $A,B,C,D$, if $A\ind_C B$ and $C\cap AB\seq D\seq C$, then $A\ind_D B$.
\item \textit{(closure)} For all closed $A,B,C$, if $C\seq A\cap B$ and $A\ind_C B$ then $AB$ is closed.
\end{enumerate}
\end{definition}

We will ultimately focus on the case when $T$ has \emph{disintegrated algebraic closure}, which is to say that the algebraic closure of any set $A\subset\M$ is the union of the algebraic closures of singleton elements in $A$. This is equivalent to the property that $AB$ is closed for any closed $A,B\subset\M$.  

\begin{example}\label{ex}
The following are examples of free amalgamation theories with disintegrated algebraic closure.
\begin{enumerate}[$(1)$]
\item Let $\cL$ be a finite relational language and let $\cK$ be a \Fraisse\ class of finite $\cL$-structures with free amalgamation. Let $T$ be the complete theory of the \Fraisse\ limit of $\cK$. Then $T$ is a free amalgamation theory, and $\acl(A)=A$ for any $A\subset\M$.

\item Let $\cL$ be the language of graphs and fix $n\geq 3$. There is a unique (up to isomorphism) countable, universal, and existentially complete $(K_n+K_3)$-free graph (where $K_n+K_3$ denotes the free amalgamation of $K_n$ and $K_3$ over a single vertex). If $T$ is the complete theory of this graph, then $T$ is a free amalgamation theory with disintegrated algebraic closure. 
\end{enumerate}
In both cases, the desired ternary relation $\ind$ is free amalgamation of relational structures: given $A,B,C\subset\M$, $A\ind_C B$ if and only if $ABC$ is the free amalgam of $AC$ and $BC$ over $C$. The verification of the axioms of Definition \ref{def} for these examples is sketched in \cite{CoFA}. In the first case, all axioms are immediate from classical \Fraisse\ theory (see, e.g., \cite{Hobook}). For the second case, the axioms rely on work of Cherlin, Shelah, and Shi \cite{CSS}, and Patel \cite{PaSOP4}. 
\end{example}

Recall  our goal is to show that if $T$ is a free amalgamation theory with disintegrated algebraic closure, then forking and dividing are the same for complete types over algebraically closed sets. Our proof strategy will follow that of \cite{Co13}, where this result is shown for the special case that $T$ is the theory of the generic $K_n$-free graph.  In particular, we will first prove a ``mixed transitivity" lemma involving dividing independence and the ternary relation $\ind$ in Definition \ref{def}. This general strategy is also used in \cite{CoKr}, \cite{dElbACFG}, and \cite{KrRa}, where the equivalence of forking and dividing for complete types (over algebraically closed sets) is established via a mixed transitivity lemma involving a stationary independence relation.

\begin{lemma}\label{lem:FAT}
Let $T$ be a free amalgamation theory with disintegrated algebraic closure. Suppose $A,B,C,D\subset\M$ are such that $D\seq C\seq B$ and $C,D$ are closed. Then
\[
\textstyle A\trt{d}_D C\mand \acl(AC)\ind_C B\mimp A\trt{d}_D B.
\]
\end{lemma}
\begin{proof}
Assume  $A\trt{d}_D C$ and $\acl(AC)\ind_C B$. Enumerate $B=b=(b_i)_{i\in I}$. Assume $I_0\seq J$ are initial segments of $I$ such that $D=(b_i)_{i\in I_0}$ and $C=c=(b_i)_{i\in J}$. Let $(b^n)_{n<\omega}$ be a $D$-indiscernible sequence, with $b^0=b$. Let $a$ enumerate $A$. We want to find $a'$ such that $a'b^n\equiv_Dab$ for all $n<\omega$. 

For $n<\omega$, let $c^n=(b^n_i)_{i\in J}$, and note that $(c^n)_{n<\omega}$ is $D$-indiscernible with $c^0=c$. There is some $I_1$ such that $I_0\seq I_1\seq J$ and, for all $m< n<\omega$ and $i\in J$, $b^m_i=b^n_i$ if and only if $i\in I_1$. Set $D'=(b_i)_{i\in I_1}$, so $D\seq D'\seq C$, and by base monotonicity for $\trt{d}$, we have $A\trt{d}_{D'} C$. Note also that $(b^n)_{n<\omega}$ and $(c^n)_{n<\omega}$ are each $D'$-indiscernible. We claim $D'$ is closed so that we may assume without loss of generality that $I_1=I_0$ and $D'=D$. 

To show $D'$ is closed, fix $d\in \acl(D')$. Since $C$ is closed, $d = b_{j_0}$ for some $j_0\in J$. Let $j_0<\dots<j_{k-1}$ be the indices of the finitely many conjugates of $b_{j_0}$ over $D'$. For every $n\in \omega$, there is some permutation $\sigma_n\in S_k$ such that $b^n_{j_i} = b_{j_{\sigma_n(i)}}$ for all $i<k$. By Pigeonhole, there exist $n<m$ such that $\sigma_n = \sigma_m$. Then $b^n_{j_i} = b^m_{j_i}$ for all $i<k$, so by indiscernibility $b^{n'}_{j_i} = b^n_{j_i}$ for all $n<n'$, and hence $j_i\in I_1$ for all $i<k$. In particular, $d = b_{j_0}\in D'$. 

Since $A\trt{d}_D C$, there is $a_*$ such that $a_*c^n\equiv_D ac$ for all $n<\omega$. Set $C_*=\acl(c^{<\omega})$. Let $e$ enumerate $\acl(a_*C_*)$. By full existence for $\ind$, there is $e'\equiv_{C_*} e$ such that $e'\ind_{C_*} b^{<\omega}$. Note that there is some $a'\equiv_{C_*}a_*$ such that $e'$ enumerates $\acl(a'C_*)$. So we have  $\acl(a'C_*)\ind_{C_*}b^{<\omega}$. For each $n<\omega$, we have $a'c^n\equiv_Da_*c^n\equiv_Dac$. By monotonicity, $\acl(a'c^n)\ind_{C_*}b^n$ for all $n<\omega$.
\medskip

\noindent\emph{Claim}: For any $n<\omega$, $C^*\cap\acl(a'c^n)b^n=c^n$.

\noindent\emph{Proof}: Since algebraic closure in $T$ is disintegrated, we have $\acl(a'c^n)=\acl(a')c^n$ and $C^*=c^{<\omega}$. So it suffices to show $c^{<\omega}\cap \acl(a')b^n=c^n$. Fix some $x\in c^{<\omega}\cap \acl(a')b^n$. There is $m<\omega$ and $i\in J$ such that $x=b^m_i$. Suppose first that $x\in \acl(a')$. Then $b^m_i\in \acl(a')\cap c^m$, which means $b_i\in \acl(a)\cap c$. Since $A\trt{d}_D C$, we have $\acl(a)\cap c\seq D$, and so $i\in I_0$. Thus $b^m_i=b^n_i\in c^n$. Finally, suppose $x\in b^n$. There is $j\in I$ such that $b^m_i=b^n_j$. It follows that $b^m_i=b^n_i$ (if $m=n$ this is trivial, and if $m\neq n$ use $b^m_i=b^n_j$ and indiscernibility). So $x=b^n_i\in c^n$.\hfill $\dashv_{\text{claim}}$
\medskip

To finish the proof, we show $a'b^n\equiv_D ab$ for all $n<\omega$. So fix $n<\omega$, and let $\sigma\in\Aut(\M/D)$ be such that $\sigma(b^n)=b$ (note that $\sigma(c^n)=c$). By the claim and freedom, we have $\acl(a'c^n)\ind_{c^n}b^n$. So $\acl(\sigma(a')c)\ind_c b$ by invariance, and since $\sigma(\acl(a'c^n))=\acl(\sigma(a')c)$. Also, we have $\sigma(a')c\equiv_D a'c^n\equiv_D ac$, and so $\sigma(a')c\equiv_c ac$. Therefore $\acl(\sigma(a')c)\equiv_c \acl(ac)$.  So we may fix tuples $e$ and $e'$ such that $\acl(ac)=ace$, $\acl(\sigma(a')c)=\sigma(a')ce'$, and $ace\equiv_c \sigma(a')ce'$. We have $\sigma(a')ce'\ind_c b$ and, by assumption, $ace\ind_c b$. Since $c\seq ace\cap b$, we may apply stationarity to conclude $aceb\equiv_c \sigma(a')ce' b$. In particular, $a'b^n\equiv_D \sigma(a')b\equiv_D ab$. 
\end{proof}

\begin{remark}
Suppose $T$ is a free amalgamation theory, witnessed by $\ind$. Given closed $A,B,C\subset\M$, with $C\seq A\cap B$, if $A\ind_C B$ then $A\trt{d}_C B$ by \cite[Lemma 7.15]{CoFA}. Therefore Lemma \ref{lem:FAT} can be seen as a weakening of transitivity for $\trt{d}$. It is worth noting that many examples of such theories are not simple (e.g. the theory of the generic $K_n$-free graph for $n\geq 3$), and hence $\trt{d}$ need not be transitive. 
\end{remark}

We now prove the main result.

\begin{theorem}\label{thm:FAT}
Let $T$ be a free amalgamation theory with disintegrated algebraic closure. Then for any $A,B,C\subset\M$, 
\[
\textstyle A\trt{f}_C B \text{ if and only if } A\trt{d}_C \acl(BC).
\]
\end{theorem}

\begin{proof}
First recall that in any theory we have $A\trt{f}_C B\Rightarrow A\trt{d}_C \acl(BC)$. So conversely, suppose $A\trt{d}_C \acl(BC)$. We want to show $A\trt{f}_C B$. By Proposition \ref{prop:PACfork}, it suffices to show $A\trt{f}_{\acl(C)}\acl(BC)$. We have $A\trt{d}_{\acl(C)} \acl(BC)$ by base monotonicity for $\trt{d}$. So, altogether, we may assume without loss of generality that $B,C$ are closed and $C\seq B$.  To show $A\trt{f}_C B$, it suffices to fix $\hat{B}\supseteq B$ and find $A'\equiv_{B}A$ such that $A'\trt{d}_C \hat{B}$. 

 Let $\ind$ witness that $T$ is a free amalgamation theory. By full existence there is $A'$ such that $A'\equiv_B A$ and $\acl(A'B)\ind_B \hat{B}$. By invariance of $\trt{d}$, we have $A'\trt{d}_C B$. By Lemma \ref{lem:FAT}, $A'\trt{d}_C \hat{B}$, as desired. 
\end{proof}

\begin{remark}
Given $n\geq 3$, let $T_n$ be the theory of the generic $K_n$-free graph. In \cite{Co13}, $\trt{d}$ is characterized for $T_n$ by purely combinatorial properties of graphs. It would be interesting to give similar descriptions of $\trt{d}$ for other theories listed in Example \ref{ex}. 
\end{remark}

\begin{question}
Does Theorem \ref{thm:FAT} hold without the assumption of disintegrated algebraic closure, or under the weaker assumption that algebraic closure is modular?
\end{question}

A possible lead toward the previous question could be recent work of Mutchnik \cite{Mutch2} which, among other things, studies a generalization of the class of free amalgamation theories.

Theorem \ref{thm:FAT} and Remark \ref{rem:acl-exist} together yield the following conclusion.

\begin{corollary}\label{cor:extbaseFAT}
If $T$ is a free amalgamation theory with disintegrated algebraic closure then all sets are extension bases for nonforking.
\end{corollary}

Finally we revisit the theory $\TOG$ from Section \ref{sec:OG}.

\begin{proposition}\label{prop:OGFAT}
$\TOG$ is a free amalgamation theory with disintegrated algebraic closure.
\end{proposition}
\begin{proof}
If $A\subset\M$ then $\acl(A)=A\cup\{0,1\}$. Therefore algebraic closure is disintegrated. Let $\ind$ denote the usual free amalgamation of relational structures. Then we have invariance, monotonicity, symmetry, full transitivity, and freedom by the standard proofs (see the discussion in~\cite[Example 3.2]{CoFA}). Full existence over closed sets follows from free amalgamation in the class $\cK$, and stationarity over closed sets follows from the description of algebraic closure and ``almost quantifier elimination". Finally, the closure axiom is trivial since algebraic closure is disintegrated. 
\end{proof}

\begin{corollary}\label{cor:OGfe}
In $\TOG$, $\trt{f}$ satisfies full existence, but $\trt{f}\neq\trt{d}$.
\end{corollary}

\section{Corrections to literature}\label{sec:corrections}

The purpose of this section is to collect and discuss some arguments in the literature which rely on the incorrect statements made in \cite[Remark 5.4]{Adgeo}. We do not claim that this list is exhaustive.

\subsection{The model completion of the empty theory}\label{sec:KrRam}

    Let $\cL$ be an arbitrary language and let $T^\emptyset_{\cL}$ be the model completion of the empty $\cL$-theory. In \cite[Proposition 3.17]{KrRa}, it is shown that $\trt{d}$ coincides with $\trt{M}$ in (any completion of) $T^\emptyset_{\cL}$. The proof quotes \cite[Remark 5.4(4)]{Adgeo} for $\trt{d}\Rightarrow\trt{M}$, which is valid in this case since algebraic closure and definable closure coincide  (see \cite[Corollary 3.11]{KrRa}). 

\subsection{Simple free amalgamation theories}

In \cite{CoFA}, preservation of algebraic closure for $\trt{d}$ is stated as Fact 7.4 and used in the proof that a free amalgamation theory is simple if and only if it is $\NTP_2$, as well as the corresponding analysis of forking independence in this case. However, it turns out that the full power of Fact 7.4 is only  required in the context of simplicity, and the remaining applications actually use weaker statements (which are true).  See the corrigenda to \cite{CoFA} for further details (available in the most recent arXiv version of \cite{CoFA}).

\subsection{Generic incidence structures}\label{sec:Tmn}

    Let $\cL=\{P,L,I\}$ where $P$ and $L$ are unary relations (for ``points" and ``lines") and $I$ is a binary relation (for ``incidences"). An \emph{incidence structure} is an $\cL$-structure in which $P$ and $L$ partition the universe and $I\seq P\times L$. In \cite{CoKr}, we study the complete theory $T_{m,n}$ of existentially closed $K_{m,n}$-free incidence structures, where $m,n\geq 1$.   The special case $T_{2,2}$ can also be viewed as the theory of existentially closed projective planes. In the paper, we define a certain ternary relation $\ter{I}$ and use it to show that $T_{m,n}$ is $\NSOP_1$ and that, over models, $\ter{I}$ coincides with Kim-independence (see \cite[Theorem 4.11]{CoKr}). We later claim in  Corollaries 4.14 and 4.24 of \cite{CoKr} that $\trt{f}$ and $\trt{d}$ are the same, and coincide with ``forcing base monotonicity" on $\ter{I}$ (in the sense of  \cite{Adgeo}). However, our arguments tacitly use \cite[Remark 5.4(3)]{Adgeo} in a few places and, as we will show below, it is in fact not true that $\trt{d}$ has algebraic extension in $T_{m,n}$ for all $m,n\geq 1$. 

    Before getting deeper into these issues, we first note that when $m=n=2$ (which is  the main case of interest), our proofs are valid because algebraic closure coincides with definable closure in this case (see \cite[Proposition 2.14]{CoKr} and the closing remarks of \cite[Section 2]{CoKr}). Our proofs are also valid  when $\min\{m,n\}=1$ because in this case $T_{m,n}$ is stable (see \cite[Section 4.5]{CoKr}).
    
    For general $m,n$, we will show  that the above claims  about $\trt{d}$ are not always true. However, it is easy to see that our proofs work with $\trt{d}$ replaced by $\trt{da}$. Thus we have that for $A,B,C\subset\M\models T_{m,n}$,
    \[
    \textstyle A\trt{f}_C B\miff A\trt{d}_C\acl(BC)\miff A\ter{I}_D B\textnormal{ for all $C\seq D\seq \acl(BC)$}.
    \] 
    Part of the proof uses \cite[Lemma 4.23]{CoKr},  which gives a mixed transitivity statement for $\trt{d}$ using a stationary independence relation (similar to Lemma \ref{lem:FAT} above). The proof tacitly assumes preservation of algebraic closure for $\trt{d}$, and so here   one  needs to also replace $\trt{d}$ with $\trt{da}$ to obtain a correct statement. 

    We also note that \cite[Proposition 4.22]{CoKr} includes the claim that $\trt{d}$ implies $\ter{I}$. While this is presented as an immediate consequence of \cite[Corollary 4.14]{CoKr} (which is false as stated), one can instead obtain this claim directly from Lemma 4.8 and Remark 4.12 of \cite{CoKr}, which are unaffected by the issues with \cite[Remark 5.4]{Adgeo}.

    Finally, we describe a counterexample to \cite[Corollary 4.14]{CoKr}, which also gives an instance where $\trt{d}$ and $\trt{f}$ disagree. We will assume some familiarity with general setup of \cite{CoKr}. Work in $T_{4,2}$, so any four points are incident a unique line, and any two lines have exactly three  points incident to them. Consider the incidence structure $(P,L;I)$ where $P=\{a_0,a_1,a_2,d_0,d_1,d_2\}$, $L=\{b_0,b_1,e\}$, and $I$ consists of $(d_0,e)$, $(d_1,e)$, $(d_i,b_j)$ for all $i$ and $j$, and $(a_i,e)$ for all $i$. Then $(P,L;I)$ is $K_{4,2}$-free, so we can freely complete it and embed it in the monster model of $T_{4,2}$. Note that $\{d_0, d_1, d_2\}$ is the unique set of three points incident to the lines $b_0$ and $b_1$, hence $d_0,d_1,d_2\in\acl(b_0,b_1)$. One can then show that $a_0a_1a_2\nter{I}_{d_0}b_0b_1$ and $a_0a_1a_2\trt{d}_\emptyset b_0b_1$, which altogether refutes \cite[Corollary 4.14]{CoKr}. Moreover, since $\trt{f}$  implies $\ter{I}$ (and  satisfies base monotonicity and  extension), we also have $a_0a_1a_2\ntrt{f}_\emptyset b_0b_1$ and thus a counterexample to \cite[Corollary 4.24]{CoKr}. We leave the verification of these details to the reader. When checking $a_0a_1a_2\trt{d}_\emptyset b_0b_1$, the key point is that every indiscernible sequence of $3$-tuples is either constant in at least two coordinates or non-constant in at least two coordinates. This explains  the need for $m=4$  in order to work with $3$-tuples. 

    Recall that $T_{4,2}$ is $\NSOP_1$ and thus, in addition to the theory $T^\eq_f$ from Section \ref{sec:GBF}, this gives another counterexample to the question from \cite{dElbACFG} discussed in the introduction. Note also that one can adjust the above configuration to obtain $\trt{f}_M\neq\trt{d}_M$ for any small model $M$ by starting with a copy of $(P,L;I)$ with no  incidences to $M$ (and taking the free completion with $M$). 

  For $m\geq 4$ and $n\geq 2$, a similar configuration can be constructed in $T_{m,n}$ using the interpretation of $T_{4,2}$ in $T_{m,n}$ described in \cite[Lemma 3.2]{CoKr}. This leaves open the case of $T_{3,2}$, which we will not pursue  here.

\subsection{Algebraically closed fields with a generic additive subgroup}\label{sec:ACFG}
For a fixed  $p>0$, let ACFG denote the model companion of the theory of algebraically closed fields of characteristic $p$ with a predicate $G$ for an additive subgroup. It turns out that the study of this theory is affected by \cite[Remark 5.4]{Adgeo} in ways very similar to the theories $T_{m,n}$ (discussed in the previous subsection). In particular, d'Elb\'{e}e  first  showed in \cite{dElbGER} that (any completion of) ACFG is properly $\NSOP_1$, and then in a sequel paper \cite{dElbACFG} it is asserted that $\trt{f}$ and $\trt{d}$ are the same and coincide with forcing base monotonicity on Kim independence. However, some arguments in \cite{dElbACFG} rely on \cite[Remark 5.4]{Adgeo}. Following discussions based on an early version of this paper, d'Elb\'{e}e has constructed an example showing that, like in $T^\eq_f$ and $T_{4,2}$, $\trt{f}$ and $\trt{d}$ need not coincide in ACFG, even over models.  On the other hand, the arguments from \cite{dElbACFG}  work with $\trt{d}$ replaced by $\trt{da}$. So it is still true that forking and dividing are the same for complete types over algebraically closed sets, and that $\trt{f}$ is the base monotonization of Kim independence. See \cite{dElb-weird} for details.

\end{document}